\documentclass[11pt]{article}%
\usepackage{a4}%
\usepackage[centertags]{amsmath}%
\usepackage{eucal,dsfont,accents,bbm,amsfonts,amssymb,amsthm,amsopn}%
\usepackage[usenames]{color}
\definecolor{citegreen}{rgb}{0,0.6,0}
\definecolor{refred}{rgb}{0.8,0,0}
\usepackage[colorlinks, citecolor=citegreen, linkcolor=refred]{hyperref}
\title{A compactness theorem for complete Ricci shrinkers}
\author{Robert Haslhofer, Reto M\"{u}ller}
\date{}
\providecommand{\abs}[1]{\lvert #1\rvert}%
\providecommand{\Abs}[1]{\left\lvert #1\right\rvert}%
\providecommand{\norm}[1]{\lVert #1\rVert}%
\providecommand{\scal}[1]{\langle #1\rangle}%
\DeclareMathOperator{\divop}{div}%
\DeclareMathOperator{\Hess}{Hess}%
\DeclareMathOperator{\Vol}{Vol}%
\DeclareMathOperator{\Rm}{Rm}%
\DeclareMathOperator{\Rc}{Rc}%
\DeclareMathOperator{\ff}{I\hspace{-0.02cm}I}%

\newcommand{\RR}{\mathbb{R}}%
\newcommand{\ZZ}{\mathbb{Z}}%
\newcommand{\NN}{\mathbb{N}}%
\newcommand{\eps}{\varepsilon}%
\newcommand{\Lap}{\triangle}%
\newcommand{\D}{\nabla}%
\newcommand{\sW}{\mathcal{W}}%
\newcommand{\ol}{\overline}%
\newcommand{\ul}{\underline}%
\newcommand{\dt}{\tfrac{\partial}{\partial t}}%
\newcommand{\ds}{\tfrac{\partial}{\partial s}}%
\newtheoremstyle{break}%
  {12pt}%
  {16pt}%
  {\itshape}%
  {}%
  {\bfseries}%
  {}%
  {\newline}%
  {\thmname{#1}\thmnumber{ #2}\thmnote{ \normalfont{(#3)}}}%

\theoremstyle{definition}%
\theoremstyle{remark}%
\newtheorem*{rem}{Remark}%
\theoremstyle{break}%
\newtheorem{lemma}{Lemma}[section]%
\newtheorem{prop}[lemma]{Proposition}%
\newtheorem{thm}[lemma]{Theorem}%
\newtheorem{cor}[lemma]{Corollary}%
\newtheorem{defn}[lemma]{Definition}%
\numberwithin{equation}{section}%
\begin{document}%
\maketitle%
\pagenumbering{arabic}%
\begin{abstract}
We prove precompactness in an orbifold Cheeger-Gromov sense of
complete gradient Ricci shrinkers with a lower bound on their
entropy and a local integral Riemann bound. We do not need any pointwise curvature assumptions, volume or diameter bounds. In dimension four, under a technical assumption, we can replace the local integral Riemann bound by an upper bound for the Euler characteristic. The proof relies on a Gauss-Bonnet with cutoff argument.
\end{abstract}
\section{Introduction}

Let us start with some background: The classical Cheeger-Gromov theorem says that every sequence of
closed Riemannian manifolds with uniformly bounded curvatures,
volume bounded below, and diameter bounded above has a
$C^{1,\alpha}$-convergent subsequence \cite{Che70, Gro99,  GW88}. The convergence is in the sense of Cheeger-Gromov, meaning $C^{1,\alpha}$-convergence of the Riemannian metrics after pulling back by suitable diffeomorphisms. Without diameter bounds, the global volume bound should be replaced
by a local volume non-collapsing assumption \cite{CGT82}, and the
appropriate notion of convergence is convergence in the pointed
Cheeger-Gromov sense. If one can also control all the derivatives of
the curvatures, e.g. in the presence of an elliptic or parabolic
equation, the convergence is smooth \cite{Ham95}. To remind the
reader about the precise definition, a sequence of complete smooth
Riemannian manifolds with basepoints $(M^n_i,g_i,p_i)$ converges to
$(M^n_{\infty},g_{\infty},p_{\infty})$ in the \emph{pointed smooth
Cheeger-Gromov sense} if there exist an exhaustion of $M_\infty$ by
open sets $U_i$ containing $p_\infty$ and smooth embeddings
$\phi_i:U_i\to M_i$ with $\phi_i(p_\infty)=p_i$ such that the pulled
back metrics $\phi_i^\ast g_i$ converge to $g_{\infty}$ in
$C^{\infty}_{\mathrm{loc}}$.\\

Now, let us describe the problem under consideration: Hamilton's Ricci flow in higher dimensions without curvature assumptions leads to the formation of intriguingly complex singularities \cite{Ham82,Ham95s}. The specific question we are concerned with is about the compactness properties of the corresponding space of singularity models. Namely, given a sequence of \emph{gradient shrinkers}, i.e.\ a sequence of
smooth, connected, complete Riemannian manifolds $(M^n_i,g_i)$
satisfying
\begin{equation}\label{1.solitoneq}
\Rc_{g_i} + \Hess_{g_i} f_i = \tfrac{1}{2}g_i
\end{equation}
for some smooth function $f_i:M\to\RR$ (called the potential), under what assumptions can we find a convergent subsequence? In the \emph{compact case}, this problem was first studied by Cao-Sesum \cite{CS07}, see also Zhang \cite{Zha06}, and Weber succeeded in removing their pointwise Ricci bounds \cite{Web08}. We have profited from these previous works and the papers by Anderson, Bando, Kasue, Nakajima and Tian about the
Einstein case \cite{And89, Nak88, BKN89, Tia90}, as well as from the papers
\cite{AndChe91,TV05a,TV05b,TV08, Uhl82}.\\

In this article, we generalize the shrinker orbifold compactness
result to the case of \emph{noncompact manifolds}. The obvious motivation
for doing this is the fact that most interesting singularity models
are noncompact, the cylinder being the most basic example. We manage to remove all volume and diameter assumptions, and we do
not need any positivity assumptions for the curvatures nor pointwise
curvature bounds (as the blow-down shrinker shows \cite{FIK03}, even the
Ricci curvature can have both signs). In fact, if the curvature is
uniformly bounded below, it is easy to pass to a smooth limit (Theorem \ref{2.smoothlimit}). The general case without positivity assumptions is much harder.\\

Having removed all other assumptions, we prove a precompactness theorem for complete Ricci shrinkers, assuming only a lower bound for the Perelman entropy and local $L^{n/2}$ bounds for the Riemann tensor (Theorem \ref{1.mainthm1}). The assumptions allow orbifold singularities to occur (these are isolated singularities modelled on $\RR^n/\Gamma$ for some finite subgroup $\Gamma\subset \mathrm{O}(n)$), and the convergence is in the \emph{pointed
orbifold Cheeger-Gromov sense}. In particular, this means that the
sequence converges in the pointed Gromov-Hausdorff sense (this is the natural notion of convergence for complete metric spaces), and that the convergence is in the smooth Cheeger-Gromov sense away from the
isolated point singularities (see Section \ref{SectionCG} for the
precise definitions).\\

Our results are most striking in dimension four. In this case, the local $L^2$ Riemann bound
is \emph{not} an a priori assumption, but we prove it modulo a technical assumption on the soliton potential (Theorem \ref{1.mainthm2}). Our proof is based on a 4d-Chern-Gauss-Bonnet with cutoff argument
(see Section \ref{SectionGB}). In particular, the key estimate of the cubic boundary term (Lemma \ref{4.RcRmlemma}) is based on a delicate use of partial integrations and soliton identities.\\

Before stating our main results, let us explain a few facts about gradient shrinkers, see Section
\ref{SectionGH} and Appendix \ref{applemmas} for proofs and further references. Associated to every
gradient shrinker $(M^n,g,f)$, there is a family of Riemannian
metrics $g(t)$, $t\in(-\infty,1)$, evolving by Hamilton's Ricci flow $\dt g(t)
=-2\Rc_{g(t)}$ with $g(0)=g$, which is self-similarly shrinking, i.e.\
$g(t)=(1-t)\phi_t^*g$ for the family of diffeomorphisms $\phi_t$
generated by $\frac{1}{1-t}\D f$, see \cite{CK,Zha09}. In this article
however, we focus on the elliptic point of view. Gradient shrinkers
always come with a natural basepoint, a point $p\in M$ where the
potential $f$ attains its minimum (such a minimum always exists
and the distance between two minimum points is bounded by a constant
depending only on the dimension). The potential grows like
one-quarter distance squared, so $2\sqrt{f}$ can be thought of as
distance from the basepoint. Moreover, the volume growth is at most
Euclidean, hence it is always possible to normalize $f$ (by adding a
constant) such that
\begin{equation}\label{1.normalizationf}
\int_M (4\pi)^{-n/2}e^{-f}dV_g = 1.
\end{equation}
Then the gradient shrinker has a well defined \emph{entropy},
\begin{equation}\label{1.defmu}
\mu(g)= \sW(g,f) = \int_M\big(\abs{\D f}_g^2 + R_g + f -
n\big)(4\pi)^{-n/2}e^{-f}dV_g>-\infty.
\end{equation}
The entropy was introduced by Perelman in his famous paper \cite{Per02} to solve the long standing problem of ruling out collapsing with bounded curvature (see \cite{BBBMP,CZ,KL08,MT} for detailed expositions of Perelman's work). For general Ricci flows, the entropy is time-dependent, but on
gradient shrinkers it is constant and finite (even without curvature
assumptions). Assuming a lower bound for the entropy is natural,
because it is non-decreasing along the Ricci flow in the compact case or
under some mild technical assumptions. Under a local scalar
curvature bound, a lower bound on the entropy gives a local
volume non-collapsing bound.\\

The main results of this article are the following two theorems.

\begin{thm}\label{1.mainthm1}
Let $(M^n_i,g_i,f_i)$ be a sequence of gradient shrinkers (with
normalization and basepoint $p_i$ as above) with entropy uniformly
bounded below, \ $\mu(g_i)\geq\ul{\mu}>-\infty$, and uniform
local energy bounds,
\begin{equation}\label{1.Rmbound}
\int_{B_{r}(p_i)}\abs{\Rm_{g_i}}_{g_i}^{n/2}dV_{g_i} \leq
E(r)<\infty, \quad \forall i,r.
\end{equation}
Then a subsequence of $(M^n_i,g_i,f_i,p_i)$ converges to an orbifold
gradient shrinker in the pointed orbifold Cheeger-Gromov sense.
\end{thm}

Here is a cute way to rephrase this theorem: \emph{The space of Ricci flow singularity models with bounded entropy and locally bounded energy is orbifold compact.}\\

In the case $n=4$, we obtain a particularly strong compactness result under a technical assumption on the potential.

\begin{thm}\label{1.mainthm2}
Let $(M^4_i,g_i,f_i)$ be a sequence of four-dimensional gradient
shrinkers (with normalization and basepoint $p_i$ as above) with
entropy uniformly bounded below, $\mu(g_i)\geq\ul{\mu}>-\infty$,
Euler characteristic bounded above, $\chi(M_i)\leq\ol{\chi}<\infty$,
and the technical assumption that the potentials do not have critical
points at large distances, more precisely
\begin{equation}\label{1.noncritf}
\abs{\D f_i}(x)\geq c>0 \qquad\text{if } d(x,p_i)\geq r_0,
\end{equation}
for some constant $r_0<\infty$. Then we have the weighted $L^2$ estimate
\begin{equation}
\int_{M_i} \abs{\Rm_{g_i}}_{g_i}^{2}e^{-f_i}dV_{g_i}\leq C(\ul{\mu},\ol{\chi},c,r_0)<\infty.
\end{equation}
In particular, the energy condition (\ref{1.Rmbound}) is satisfied
and by Theorem \ref{1.mainthm1} a subsequence converges in the
pointed orbifold Cheeger-Gromov sense.
\end{thm}

As explained above, to appreciate our theorems it is most important
to think about the assumptions that we do \emph{not} make.

\begin{rem}
The technical assumption (\ref{1.noncritf}) is satisfied in particular
if the scalar curvature satisfies
\begin{equation}\label{1.strictlyless}
R_{g_i}(x)\leq \alpha d(x,p_i)^2+C
\end{equation}
for some $\alpha<\tfrac{1}{4}$. The scalar curvature grows at most
like one-quarter distance squared and the average scalar curvature on $2\sqrt{f}$-balls is bounded by $n/2$ (see Section \ref{SectionGH} and Appendix \ref{applemmas}), so
the technical assumption is rather mild. However, it would be very desirable
to remove (or prove) it. Of course, (\ref{1.noncritf}) would also
follow from a diameter bound.
\end{rem}

In this article, the potentials of the gradient shrinkers play a
central role in many proofs. In particular, we can view (a
perturbation of) $f$ as a Morse function, use $e^{-f}$ as weight or
cutoff function and use balls defined by the distance $2\sqrt{f}$
instead of the Riemannian distance. This has the great advantage,
that we have a formula for the second fundamental form in the
Gauss-Bonnet with boundary argument.\\

There are very deep and interesting other methods that yield
comparable results, in particular the techniques developed by
Cheeger-Colding-Tian in their work on the structure of spaces with
Ricci curvature bounded below (see \cite{Che03} for a nice survey)
and the nested blowup and contradiction arguments of Chen-Wang
\cite{CW09}. Finally, let us mention the very interesting recent
paper by Song-Weinkove \cite{SW10}.\\

This article is organized as follows. In Section \ref{SectionGH},
we collect and prove some properties of gradient shrinkers. In
Section \ref{SectionCG}, we prove Theorem \ref{1.mainthm1}. Finally, we prove
Theorem \ref{1.mainthm2} in Section \ref{SectionGB} using the
Chern-Gauss-Bonnet theorem for manifolds with boundary and carefully
estimating the boundary terms. We would like to point out that the
Sections \ref{SectionCG} and \ref{SectionGB} are completely
independent of each other and can be read in any
order.

\paragraph{Acknowledgments:} We greatly thank Tom Ilmanen for suggesting this problem. We also thank him and Carlo Mantegazza for very
interesting discussions, and the anonymous referee for useful suggestions that greatly helped to improve the exposition. The first author was partially supported by
the Swiss National Science Foundation, the research of the second
author was supported by the Italian FIRB Ideas ``Analysis
and Beyond" and by The Leverhulme Trust.

\section{Some properties of gradient shrinkers}\label{SectionGH}

Let us start by collecting some basic facts about gradient
shrinkers (for a recent survey about Ricci solitons, see \cite{Cao09}). Tracing the soliton equation,
\begin{equation}
R_{ij}+\D_i\D_j f =\tfrac{1}{2}g_{ij},\label{2.soleqn}
\end{equation}
gives
\begin{equation}
R+\Lap f =\tfrac{n}{2}.\label{2.tracedsoleqn}
\end{equation}
Using the contracted second Bianchi identity, inserting the soliton equation (\ref{2.soleqn}), and commuting the derivatives, we compute
\begin{equation}
\tfrac{1}{2}\D_iR=\D_iR-\D_jR_{ij}
=-\D_i\D_j\D_jf+\D_j\D_i\D_jf=R_{ik}\D_kf.\label{2.bianchi}
\end{equation}
As observed by Hamilton, from this formula and equation (\ref{2.soleqn}), it follows that
\begin{equation}
C_1(g):= R+\abs{\D f}^2-f\label{2.auxiliaryeqn}
\end{equation}
is constant (note that we always assume that our manifold is connected). By (\ref{2.soleqn}), the Hessian of $f$ is uniquely
determined by $g$. Thus, the potential has the form
$f(x,y)=\tilde{f}(x)+\tfrac{1}{4}\abs{y-y_0}^2$ after splitting
$M\cong \tilde{M}\times \RR^k$ isometrically. Note that the constant $C_1(g)$ and also the normalization (\ref{1.normalizationf}) do not depend on the point $y_0\in\RR^k$. It follows that $\tilde{f}$ is completely determined by fixing the normalization (\ref{1.normalizationf}), and that $C_1(g)$ is independent of $f$ after fixing this normalization. Gradient shrinkers always have nonnegative scalar
curvature,
\begin{equation}
R\geq 0.\label{2.scalarpos}
\end{equation}
This follows from the elliptic equation
\begin{equation}
R+\scal{\D f,\D R}=\Lap R+2\abs{\Rc}^2\label{2.ellipteqR}
\end{equation}
by the maximum principle, see \cite{Zha09} for a proof in the
noncompact case without curvature assumptions. Equation
(\ref{2.ellipteqR}) is the shrinker version of the evolution
equation $\dt R=\Lap R +2\abs{\Rc}^2$ under Ricci flow.\\

The following two lemmas show, that the shrinker potential $f$ grows
like one-quarter distance squared and that gradient shrinkers have
at most Euclidean volume growth.

\begin{lemma}[Growth of the potential]\label{2.growthlemma}
Let $(M^n,g,f)$ be a gradient shrinker with $C_1=C_1(g)$ as in
(\ref{2.auxiliaryeqn}).  Then there exists a point $p\in M$ where
$f$ attains its infimum and $f$ satisfies the quadratic growth
estimate
\begin{equation}\label{2.quadgrowth}
\tfrac{1}{4}\big(d(x,p)-5n\big)_{\!+}^{2} \leq f(x)+C_1 \leq
\tfrac{1}{4}\big(d(x,p)+\sqrt{2n}\big)^2
\end{equation}
for all $x\in M$, where $a_+ := \max\{0,a\}$. If
$p_1$ and $p_2$ are two minimum points, then their distance is bounded by
\begin{equation}\label{2.distbound}
d(p_1,p_2)\leq 5n+\sqrt{2n}.
\end{equation}
\end{lemma}

\begin{lemma}[Volume growth]\label{2.ballslemma}
There exists a constant $C_2=C_2(n)<\infty$ such that every gradient shrinker $(M^n,g,f)$ with $p\in M$ as in Lemma \ref{2.growthlemma} satisfies the volume growth estimate
\begin{equation}\label{2.ballseqn}
\Vol B_r(p)\leq C_2r^n.
\end{equation}
\end{lemma}

The proofs are small but crucial improvements of the proofs by Cao-Zhou and Munteanu \cite{CZ09, Mun09}. In fact, their results are not strong enough for our purpose for which it is necessary, in particular, to remove the dependence on the geometry on a unit ball in Theorem 1.1 of \cite{CZ09} and to show that the constant in the
volume growth estimate can be chosen uniformly for a sequence of shrinkers. In order to keep this section compact, we moved the proofs of both lemmas to Appendix \ref{applemmas}.\\

\emph{From now on, we fix a point $p\in M$ where $f$ attains its
minimum.}\\

By Lemma \ref{2.growthlemma} and Lemma
\ref{2.ballslemma}, any function $\varphi$ that satisfies the growth
estimate
\begin{equation}\label{2.expgrow}
\abs{\varphi(x)}\leq Ce^{\alpha d(x,p)^2} \qquad \textrm{for some}\;
\alpha<\tfrac{1}{4}
\end{equation}
is integrable with respect to the measure
$e^{-f}dV$. In particular, the integral $\int_M e^{-f}dV$ is finite and $f$
can be normalized (by adding a constant if necessary) to satisfy the
normalization constraint (\ref{1.normalizationf}).\\

\emph{From now on, we will fix the normalization (\ref{1.normalizationf})}.\\

Let us now explain the logarithmic Sobolev inequality, compare with
Carrillo-Ni \cite{CN08}.
Any polynomial in $R,f,\abs{\D
f},\Lap f$ is integrable with respect to the measure
$e^{-f}dV$. Indeed, using one after another (\ref{2.scalarpos}), $0\leq\abs{\D f}^2$, (\ref{2.auxiliaryeqn}), and Lemma \ref{2.growthlemma}, we compute
\begin{equation}\label{2.scalarest}
0\leq R(x)\leq R(x)+\abs{\D f}^2(x)=f(x)+C_1\leq\tfrac{1}{4}\big(d(x,p)+\sqrt{2n}\big)^2,
\end{equation}
and using furthermore (\ref{2.tracedsoleqn}) this implies
\begin{equation}
-\tfrac{n}{2}\leq -\Lap f(x)\leq -\tfrac{n}{2}+\tfrac{1}{4}\big(d(x,p)+\sqrt{2n}\big)^2.
\end{equation}
So, any polynomial in $R,f,\abs{\D
f},\Lap f$ has at most polynomial growth and thus in particular satisfies the growth estimate (\ref{2.expgrow}). It follows that the entropy
\begin{align}
\mu(g):=\sW(g,f)&=\int_M\big(\abs{\D f}^2+R+f-n\big)\label{2.defnmu}
(4\pi)^{-n/2}e^{-f}dV
\end{align}
is well defined. To obtain another expression for the entropy, we will use the partial integration formula
\begin{equation}
\int_M\Lap f e^{-f} dV=\int_M\abs{\D f}^2 e^{-f} dV,\label{2.piform}
\end{equation}
which is justified as follows: Let
$\eta_r(x):=\eta(d(x,p)/r)$, where $0\leq\eta\leq 1$ is a cutoff
function such that $\eta(s)=1$ for $s\leq 1/2$ and $\eta(s)=0$ for
$s\geq 1$. Then
\begin{equation*}
\int_M\eta_r \Lap f e^{-f}dV=\int_M\eta_r \abs{\D f}^2
e^{-f}dV-\int_M\scal{\D \eta_r,\D f} e^{-f}dV.
\end{equation*}
Now, using the estimates for $f,\abs{\D f}$ and the volume growth, we see
that
\begin{equation*}
\int_M\abs{\D \eta_r}\abs{\D f} e^{-f}dV\leq C\int_{B_r(p)\setminus
B_{r/2}(p)}\tfrac{1}{r} d(x,p) e^{-\tfrac{(d(x,p)-5n)^2}{4}}dV\\
\leq Cr^ne^{-\tfrac{(r/2-5n)^2}{4}}
\end{equation*}
converges to zero for $r\to\infty$, and (\ref{2.piform}) follows from the dominated convergence theorem. Moreover, note that (\ref{2.tracedsoleqn}) and (\ref{2.auxiliaryeqn}) imply the formula
\begin{equation}
2\Lap f-\abs{\D f}^2+R+f-n=-C_1.
\end{equation}
Putting everything together, we conclude that
\begin{align}
\mu(g)=\int_M\big(2\Lap f-\abs{\D f}^2+R+f-n\big)
(4\pi)^{-n/2}e^{-f}dV=-C_1(g),
\end{align}
where we also used the normalization
(\ref{1.normalizationf}) in the last step. In other words, the \emph{auxiliary
constant} $C_1(g)$ of the gradient shrinker is minus the \emph{Perelman
entropy}. 
Carrillo-Ni
made the wonderful observation that Perelman's logarithmic Sobolev
inequality holds even for noncompact shrinkers without curvature
assumptions \cite[Thm. 1.1]{CN08}, i.e.
\begin{equation}
\inf\sW(g,\tilde{f}) \geq \mu(g),\label{2.logsobolev}
\end{equation}
where the infimum is taken over all
$\tilde{f}:M\to\RR\cup\{+\infty\}$ such that
$\tilde{u}=e^{-\tilde{f}/2}$ is smooth with compact support and
$\int_M \tilde{u}^2 dV=(4\pi)^{n/2}$. Essentially, this follows from
$\Rc_f=\Rc+\Hess f\geq 1/2$ and the Bakry-Emery theorem
\cite[Thm.~21.2]{Vil09}.\\

\begin{rem}
In fact, equality holds in
(\ref{2.logsobolev}), which can be seen as follows. First observe
that, as a function of $g$ and $\tilde{u}$,
\begin{equation}
\sW(g,\tilde{u})=(4\pi)^{-n/2}\int_M\left(4\abs{\D\tilde{u}}^2
+(R-n)\tilde{u}^2-\tilde{u}^2\log\tilde{u}^2\right)dV,
\end{equation}
and that one can take the infimum over all properly normalized
Lipschitz functions $\tilde{u}$ with compact support. Now, the
equality follows by approximating $u=e^{-f/2}$ by
$\tilde{u}_r:=C_r\eta_ru$, with $\eta_r$ as above and with constants
\begin{equation}
C_r=\sqrt{\frac{(4\pi)^{n/2}}{\int_M\eta_r^2 u^2 dV}}\searrow{1}
\end{equation}
to preserve the normalization. Indeed, arguing as before we see that
\begin{equation}
\begin{aligned}
\int_M(R-n)\eta_r^2u^2 &\to\int_M(R-n)u^2, &
\int_MC_r^2\eta_r^2u^2\log u^2 &\to\int_M u^2\log u^2,\\
\int_M\abs{\D(\eta_r u)}^2 &\to\int_M\abs{\D u}^2, &
\int_MC_r^2\eta_r^2\log(C_r^2\eta_r^2)u^2 &\to 0,
\end{aligned}
\end{equation}
which together yields $\sW(g,\tilde{u}_r)\to\sW(g,u)$.
\end{rem}

From
Perelman's logarithmic Sobolev inequality (\ref{2.logsobolev}) and
the local bounds for the scalar curvature (\ref{2.scalarest}), we get the following
non-collapsing lemma.

\begin{lemma}[Non-collapsing]\label{2.noncollapslemma} There exists a function
$\kappa(r)=\kappa(r,n,\ul{\mu})>0$ such that for every gradient
shrinker $(M^n,g,f)$ (with basepoint $p$ and normalization as
before) with entropy bounded below, $\mu(g)\geq\ul{\mu}$, we have
the lower volume bound $\Vol B_\delta(x)\geq\kappa(r)\delta^n$ for
every ball $B_\delta(x)\subset B_r(p)$, $0<\delta\leq 1$.
\end{lemma}

The proof is strongly related to Perelman's proof for finite-time Ricci flow singularities (see Kleiner-Lott \cite[Sec.~13]{KL08} for a nice and detailed exposition), and can be found in Appendix \ref{applemmas}. Given a lower bound $\mu(g)\geq\ul{\mu}$, we  also get an upper bound
$\mu(g) \leq \ol{\mu}=\ol{\mu}(\ul{\mu},n)$ using
$\tilde{u}=c^{-1/2}\eta(d(x,p))$ as test function. Of course, the
conjecture is $\mu(g)\leq 0$ even for noncompact shrinkers without
curvature assumptions.\\

Equipped with the above lemmas, we can now easily prove the non-collapsed pointed Gromov-Hausdorff convergence in the general case, and the pointed smooth Cheeger-Gromov convergence in the case where the curvature is uniformly bounded below.

\begin{thm}[Non-collapsed Gromov-Hausdorff convergence]\label{2.GHlimit}
Let $(M_i^n,g_i,f_i)$ be a sequence of gradient shrinkers with
entropy uniformly bounded below, $\mu(g_i)\geq\ul{\mu}>-\infty$, and
with basepoint $p_i$ and normalization as before. Then the sequence is volume non-collapsed at finite distances from the basepoint and a subsequence $(M_i,d_i,p_i)$ converges to a complete metric space in the pointed Gromov-Hausdorff sense.
\end{thm}

\begin{proof}
The first part is Lemma \ref{2.noncollapslemma}. For the second part, to find a subsequence that converges in the pointed Gromov-Hausdorff sense, it suffices to find uniform bounds $N(\delta,r)$ for the number of disjoint $\delta$-balls that
fit within an $r$-ball centered at the basepoint \cite[Prop.~5.2]{Gro99}. Assume $\delta\leq 1$ without loss of generality. By Lemma \ref{2.ballslemma} the ball $B_r(p)$ has volume at most $C_2r^n$, while by Lemma \ref{2.noncollapslemma} each ball $B_{\delta}(x)\subset B_r(p)$ has volume at least $\kappa\delta^n$. Thus, there can be at most $N(\delta,r)=C_2 r^n/\kappa\delta^n$ disjoint $\delta$-balls in $B_r(p)$.
\end{proof}

\begin{rem} Alternatively, the Gromov-Hausdorff convergence also follows from the volume comparison theorem of Wei-Wylie for the Bakry-Emery Ricci tensor \cite{WW09}, using the estimates for the soliton potential from this section. This holds even without entropy and energy bounds, but in that case the limit can be collapsed and very singular.
\end{rem}

\begin{thm}[Smooth convergence in the curvature bounded below case]\label{2.smoothlimit}
Let $(M_i^n,g_i,f_i)$ be a sequence of gradient shrinkers (with
basepoint $p_i$ and normalization as before) with entropy uniformly
bounded below, $\mu(g_i)\geq\ul{\mu}>-\infty$, and curvature
uniformly bounded below, $\Rm_{g_i}\geq \ul{K}>-\infty$. Then a
subsequence $(M_i,g_i,f_i,p_i)$ converges to a gradient shrinker $(M_\infty,g_\infty,f_\infty,p_\infty)$ in the pointed smooth Cheeger-Gromov sense 
(i.e.~there exist an exhaustion of $M_\infty$ by
open sets $U_i$ containing $p_\infty$ and smooth embeddings
$\phi_i:U_i\to M_i$ with $\phi_i(p_\infty)=p_i$ such that $(\phi_i^\ast g_i,\phi_i^\ast f_i)$ converges to $(g_{\infty},f_\infty)$ in $C^{\infty}_{\mathrm{loc}}$).
\end{thm}

\begin{proof}
Recall the following Cheeger-Gromov compactness theorem from the very beginning of the introduction: For every sequence $(M^n_i,g_i,p_i)$ of complete Riemannian manifolds with uniform local bounds for the curvatures and all its derivatives,
\begin{equation}
\sup_{B_r(p_i)}\abs{\D^k\Rm_{g_i}}\leq C_k(r),\label{2.cond1}
\end{equation}
and with a uniform local volume-noncollapsing bound around the basepoint,
\begin{equation}
\Vol_{g_i} B_1(p_i)\geq\kappa,\label{2.cond2}
\end{equation}
we can find a subsequence that converges to a limit $(M_\infty,g_\infty,p_\infty)$ in the pointed smooth Cheeger-Gromov sense.\\
Moreover, if we have also uniform local bounds for the shrinker potential and all its derivatives,
\begin{equation}
\sup_{B_r(p_i)}\abs{\D^k f_i}\leq C_k(r),\label{2.cond3}
\end{equation}
then by passing to another subsequence if necessary the functions $\phi_i^\ast f_i$ will also converge to some function $f_\infty$ in $C^{\infty}_{\mathrm{loc}}$ (the embeddings $\phi_i:U_i\to M_i$ come from the pointed Cheeger-Gromov convergence). From the very definition of smooth convergence it is clear that the shrinker equation will pass to the limit, i.e.~that $(M_\infty,g_\infty,f_\infty)$ will be a gradient shrinker. Thus, it remains to verify (\ref{2.cond1}), (\ref{2.cond2}), and (\ref{2.cond3}) for our sequence of shrinkers.\\
By Lemma \ref{2.noncollapslemma}, we have uniform local volume non-collapsing, in particular condition (\ref{2.cond2}) is satisfied.
From (\ref{2.scalarest}) we have uniform local bounds for the scalar curvature, and together with the
assumption $\Rm_{g_i}\geq \ul{K}$ this gives uniform local Riemann bounds,
\begin{equation}
\sup_{B_r(p_i)}\abs{\Rm_{g_i}}\leq C_0(r).\label{2.given1}
\end{equation}
From (\ref{2.scalarest}) and the bounds $\ul{\mu}\leq-C_1(g_i)\leq\ol{\mu}$, we get local $C^1$ bounds for $f_i$,
\begin{equation}
\sup_{B_r(p_i)}\abs{ f_i}\leq C_0(r),\qquad \sup_{B_r(p_i)}\abs{\D f_i}\leq C_1(r).\label{2.given2}
\end{equation}
Finally, by some very well known arguments, we can bootstrap the elliptic system
\begin{equation}
\begin{split}
\Lap\Rm&=\D f*\D\Rm+\Rm+\Rm*\Rm,\\
\Lap f &= \tfrac{n}{2}-R,
\end{split}
\end{equation}
starting from (\ref{2.given1}) and (\ref{2.given2}) to arrive at (\ref{2.cond1}) and (\ref{2.cond3}). Here, the second equation is just the traced soliton equation (\ref{2.tracedsoleqn}), while the first equation is obtained from the soliton equation (\ref{2.soleqn}) and the Bianchi identity as follows:
\begin{align}
\D_p\D_p R_{ijk\ell}&=-\D_p\D_k R_{ij\ell p}-\D_p\D_\ell R_{ijpk}\nonumber\\
&=-\D_k\D_p R_{ij\ell p}-\D_\ell\D_p R_{ijpk}+(\Rm\ast\Rm)_{ijk\ell}\nonumber\\
&=\D_k(\D_iR_{j\ell}-\D_j R_{i\ell})+\D_\ell(\D_jR_{ik}-\D_i R_{jk})+(\Rm\ast\Rm)_{ijk\ell}\nonumber\\
&=\D_k(R_{ji\ell p}\D_p f)+\D_\ell(R_{ijkp}\D_p f)+(\Rm\ast\Rm)_{ijk\ell}\nonumber\\
&=(\D f*\D\Rm+\Rm+\Rm*\Rm)_{ijk\ell}.\label{2.ellder}
\end{align}
Here, we used the Bianchi identity and the commutator rule in the first three lines and in the fourth and fifth line we used the soliton equation. This finishes the proof of the theorem.
\end{proof}

\begin{rem}
The more interesting case without positivity assumptions is treated in Section \ref{SectionCG}. A related simple and well known example for singularity formation is
the following. Consider the Eguchi-Hanson metric $g_{\mathrm{EH}}$
\cite{EH79}, a Ricci-flat metric on $TS^2$ which is asymptotic to
$\RR^4/\ZZ_2$ (remember that the unit tangent bundle of the 2-sphere
is homeomorphic to $S^3/\ZZ_2$). Then
$g_i:=\tfrac{1}{i}g_{\mathrm{EH}}$ is a sequence of Ricci-flat
metrics, that converges to $\RR^4/\ZZ_2$ in the orbifold
Cheeger-Gromov sense. In particular, an orbifold singularity
develops as the central 2-sphere (i.e. the zero section) shrinks to
a point. For the positive K{\"a}hler-Einstein case, see Tian \cite{Tia90}, in particular Theorem 7.1.
\end{rem}

\begin{rem}
For a sequence of gradient shrinkers with entropy uniformly bounded below, by Lemma \ref{2.growthlemma} and Lemma \ref{2.ballslemma},
$(4\pi)^{-n/2}e^{-f_i}dV_{g_i}$ is a sequence of uniformly tight
probability measures. Thus, a subsequence of
$(M_i,d_i,e^{-f_i}dV_{g_i},p_i)$ converges to a pointed measured complete metric space
$(M_{\infty},d_{\infty},\nu_{\infty},p_{\infty})$ in the pointed
\emph{measured} Gromov-Hausdorff sense. By (\ref{2.scalarest}), Lemma
\ref{2.growthlemma} and a Gromov-Hausdorff version of the
Arzela-Ascoli theorem, there exists a continuous limit function
$f_\infty:M_\infty\to\RR$. It is an interesting question, if
$\nu_{\infty}$ equals $e^{-f_\infty}$ times the Hausdorff measure.
\end{rem}

\begin{rem}
It follows from the recent work of Lott-Villani and Sturm that
the condition $\Rc_f\geq 1/2$ is preserved in a weak sense
\cite{Vil09}.
\end{rem}

\section{Proof of orbifold Cheeger-Gromov convergence}\label{SectionCG}

In this section, we prove Theorem \ref{1.mainthm1}. For convenience of the reader, we will also explain some steps that are based on well known compactness techniques.\\

The structure of the proof is the following: First, we show that we have a uniform estimate for the local Sobolev constant (Lemma \ref{3.sobconst}). Using this, we prove the $\eps$-regulartiy Lemma \ref{3.epsreglemma}, which says that we get uniform bounds for the curvatures on balls with small energy. We can then pass to a smooth limit away from locally finitely many singular points using in particular the energy assumption (\ref{1.Rmbound}) and the $\eps$-regulartiy lemma.  This limit can be completed as a metric space by adding locally finitely many points. Finally, we prove that the singular points are of smooth orbifold type.\\

We start by giving a precise definition of the convergence.

\begin{defn}[Orbifold Cheeger-Gromov convergence]\label{3.defconv}
A sequence of gradient shrinkers $(M^n_i,g_i,f_i,p_i)$
converges to an orbifold gradient shrinker $(M^n_\infty,g_\infty,f_\infty,p_\infty)$ in the pointed orbifold Cheeger-Gromov sense, if the following properties hold.
\begin{enumerate}
\item There exist a locally finite set $S\subset M_\infty$, an exhaustion of $M_\infty\setminus S$ by open sets $U_i$ and smooth embeddings $\phi_i:U_i \to M_i$, such that $(\phi_i^*g_i,\phi_i^*f_i)$ converges to $(g_\infty,f_\infty)$ in $C^\infty_{\textrm{loc}}$ on $M_\infty\setminus S$.
\item The maps $\phi_i$ can be extended to pointed Gromov-Hausdorff approximations yielding a convergence $(M_i,d_i,p_i)\to(M_\infty,d_\infty,p_\infty)$ in the pointed Gromov-Hausdorff sense.
\end{enumerate}
\end{defn}

Here, an orbifold gradient shrinker is a complete metric space that is a
smooth gradient shrinker away from locally finitely many singular
points. At a singular point $q$, $M_\infty$ is modeled on $\RR^n/\Gamma$ for some finite subgroup $\Gamma\subset O(n)$ and there is an associated covering $\RR^n\supset B_\varrho(0)\setminus\{0\}\stackrel{\pi}{\to} U\setminus\{q\}$ of some neighborhood $U\subset M_\infty$ of $q$ such that $(\pi^*g_\infty,\pi^*f_\infty)$ can be extended smoothly to a gradient shrinker over the origin.\\

For the arguments that follow, it will be very important to have a uniform local Sobolev constant that works simultaneously for all shrinkers in our sequence.

\begin{lemma}[Estimate for the local Sobolev constant]\label{3.sobconst}
There exist $C_S(r)=C_S(r,n,\ul{\mu})<\infty$ and $\delta_0(r)=\delta_0(r,n,\ul{\mu})>0$ such that for every gradient shrinker $(M^n,g,f)$ (with basepoint $p$ and normalization as
before) with $\mu(g)\geq\ul{\mu}$, we have
\begin{equation}
\norm{\varphi}_{L^{\frac{2n}{n-2}}}\leq C_S(r) \norm{\D \varphi}_{L^2}
\end{equation}
for all balls $B_\delta(x)\subset B_r(p), 0<\delta\leq\delta_0(r)$ and all functions $\varphi\in C^1_c(B_\delta(x))$.
\end{lemma}

\begin{proof}
The main point is that the estimate for the local Sobolev constant will follow from the noncollapsing and the volume comparison for the Bakry-Emery Ricci tensor. The detailed argument goes as follows:\\
The first reduction is that it suffices to control the optimal constant $C_{1}(B)$ in the $L^1$-Sobolov inequality,
\begin{equation}
\norm{\psi}_{L^{\frac{n}{n-1}}(B)}\leq C_{1}(B) \norm{\D \psi}_{L^1(B)},\label{3.l1sob}
\end{equation}
for all $\psi\in C^1_c(B)$, where in our case $B$ will always be an open ball in a Riemannian manifold. Indeed, applying (\ref{3.l1sob}) for $\psi=\varphi^{(2n-2)/(n-2)}$ and using H{\"o}lder's inequality we can compute
\begin{align*}
\Big(\norm{\varphi}_{L^{\frac{2n}{n-2}}(B)}\Big)^{\!\tfrac{2n-2}{n-2}}
&\leq \tfrac{2n-2}{n-2} C_{1}(B) \norm{\varphi^{n/(n-2)}\D\varphi}_{L^1(B)}\\
&\leq \tfrac{2n-2}{n-2} C_{1}(B) \Big(\norm{\varphi}_{L^{\frac{2n}{n-2}}(B)}\Big)^{\!\tfrac{n}{n-2}} \norm{\D \varphi}_{L^2(B)},
\end{align*}
so the $L^1$-Sobolov inequality (\ref{3.l1sob}) implies the $L^2$-Sobolov inequality
\begin{equation}
\norm{\varphi}_{L^{\frac{2n}{n-2}}(B)}\leq C_{2}(B) \norm{\D \varphi}_{L^2(B)},
\end{equation}
with $C_2(B)=\tfrac{2n-2}{n-2} C_{1}(B)$.
Next, it is a classical fact, known under the name Federer-Fleming theorem, that the optimal constant $C_{1}(B)$ in (\ref{3.l1sob}) is equal to the optimal constant $C_I(B)$ in the isoperimetric inequality,
\begin{equation}
\abs{\Omega}^{\tfrac{n-1}{n}}\leq C_I(B)\abs{\partial \Omega},
\end{equation}
for all regions $\Omega\Subset B$ with $C^1$-boundary. Third, by a theorem of Croke \cite[Thm.~11]{Cro80}, the isoperimetric constant can be estimated by
\begin{equation}
C_I(B)\leq C(n)\omega(B)^{-\tfrac{n+1}{n}},
\end{equation}
where $C(n)<\infty$ is an explicit constant whose value we do not need and $\omega(B)$ is the visibility angle defined as
\begin{equation}
\omega(B)=\inf_{y \in B}\abs{U_y}/\abs{S^{n-1}},\label{3.visang}
\end{equation}
where $U_y=\{v\in T_yB\; ;\; \abs{v}=1,\textrm{the geodesic}\; \gamma_v \;\textrm{is minimizing up to}\; \partial B\}.$\\
Putting everything together, to finish the proof of our lemma it suffices to find a lower bound for the visibility angle (\ref{3.visang}) for $B=B_\delta(x)\subset B_r(p)$ inside a shrinker for $\delta\leq \delta_0(r)$, where $\delta_0(r)$ will be chosen sufficiently small later. To find such a lower bound, we will use the volume comparison theorem for the Bakry-Emery Ricci tensor due to Wei-Wylie \cite{WW09} which we now explain:\\
Fix $y\in M$, use exponential polar coordinates around $y$ and write $dV=\mathcal{A}(r,\theta)dr\wedge d\theta$ for the volume element, where $d\theta$ is the standard volume element on the unit sphere $S^{n-1}$. Let $\mathcal{A}_{f}(r,\theta)=\mathcal{A}(r,\theta)e^{-f}$. Note that $\Rc_f=\Rc+\D^2 f\geq 0$ by the soliton equation. The angular version of the volume comparison theorem for the Bakry-Emery Ricci tensor \cite[Thm 1.2a]{WW09} says that if in addition $\abs{\D f}\leq a$ on $B_R(y)$ then
\begin{equation}
\frac{\int_0^R\mathcal{A}_{f}(s,\theta)ds}{\int_0^r\mathcal{A}_{f}(s,\theta)ds}\leq
e^{aR}\left(\frac{R}{r}\right)^n.
\end{equation}
for $0<r\leq R$. If we have moreover $\max_{B_R(y)}f\leq \min_{B_R(y)}f+ b$ then this implies
\begin{equation}
{\int_0^R\mathcal{A}(s,\theta)ds}\leq
e^{aR+b}\left(\frac{R}{r}\right)^n {\int_0^r\mathcal{A}(s,\theta)ds},
\end{equation}
and finally by sending $r$ to zero we obtain the form of the volume comparison estimate that we will actually use, namely
\begin{equation}
{\int_0^R\mathcal{A}(s,\theta)ds}\leq \tfrac{1}{n}e^{aR+b}R^n.\label{volcompest}
\end{equation}
In our application everything will stay inside a ball $B_{r+1}(p)$ around the basepoint of the soliton, so by (\ref{2.scalarest}) we can take $a:=\tfrac{1}{2}(r+1+\sqrt{2n})$ and $b:=a^2$.\\
Now, coming back to actually estimating the visibility angle of $B_\delta(x)\subset B_r(p)$, we let $y\in B_\delta(x)$ and apply the above ideas. Since the volume is computed using exponential polar coordinates around $y$, we have the estimate
\begin{equation}
\abs{B_1(x)}-\abs{B_\delta(x)}\leq\int_{U_y}\int_0^{1+\delta}\mathcal{A}(s,\theta)dsd\theta,
\end{equation}
where $U_y$ denotes the set of unit tangent vectors in whose direction the geodesics are minimizing up to the boundary of $B_\delta(x)$. Using (\ref{volcompest}), we can estimate this by
\begin{equation}
\abs{B_1(x)}-\abs{B_\delta(x)}\leq \tfrac{1}{n}e^{2a+b}\abs{U_y}(1+\delta)^n,
\end{equation}
and minimizing over $y\in B_{\delta}(x)$ we obtain the inequality
\begin{equation}
\abs{B_1(x)}-\abs{B_\delta(x)}\leq C\omega(B_\delta(x))(1+\delta)^n\label{vis1}
\end{equation}
with $C=C(r,n)= \tfrac{1}{n}e^{2a+b}\abs{S^{n-1}}$. Moreover, using (\ref{volcompest}) again, we obtain the upper bound
\begin{equation}
\abs{B_\delta(x)}\leq\int_{S^{n-1}}\int_0^{\delta}\mathcal{A}(s,\theta)dsd\theta\leq C\delta^n.\label{vis2}
\end{equation}
Finally, we have the lower volume bound
\begin{equation}
\abs{B_1(x)}\geq \kappa,\label{vis3}
\end{equation}
for $\kappa=\kappa(r+1,n,\ul{\mu})$ from Lemma \ref{2.noncollapslemma}. If we now choose $\delta_0=\delta_0(r,n,\ul{\mu})=(\kappa/2C)^{1/n}$, then putting together (\ref{vis1}), (\ref{vis2}) and (\ref{vis3}) gives the lower bound
\begin{equation}
\omega(B_\delta(x))\geq\frac{\kappa}{2^{n+1}C}
\end{equation}
for the visibility angle for $\delta\leq\delta_0$, and this finishes the proof of the lemma.
\end{proof}

Using the uniform estimate for the local Sobolev constant, we obtain the following $\eps$-regularity lemma.

\begin{lemma}[$\eps$-regularity]
\label{3.epsreglemma}
There exist $\eps_1(r)=\eps_1(r,n,\ul{\mu})>0$ and $K_\ell(r)=K_\ell(r,n,\ul{\mu})<\infty$ such that for every gradient shrinker $(M^n,g,f)$ (with basepoint $p$ and normalization as
before) with $\mu(g)\geq\ul{\mu}$ and for every ball $B_\delta(x)\subset B_r(p), 0<\delta\leq\delta_0(r)$, we have the implication
\begin{equation}\label{3.epsreg}
\norm{\Rm}_{L^{n/2}(B_{\delta}(x))}\leq \eps_1(r)
\Rightarrow
\sup_{B_{\delta/4}(x)}\abs{\D^\ell \Rm} \leq
\frac{K_\ell(r)}{\delta^{2+\ell}}\norm{\Rm}_{L^{n/2}(B_{\delta}(x))}.
\end{equation}
\end{lemma}

\begin{proof}
The gradient shrinker version of the evolution equation of the Riemann tensor under Ricci flow, $\dt\Rm=\Lap\Rm+Q(\Rm)$, is the elliptic equation
\begin{equation}
\Lap\Rm=\D f*\D\Rm+\Rm+\Rm*\Rm.\label{3.ellequ}
\end{equation}
Here, we used (\ref{2.soleqn}) to eliminate $\D^2f$ in $L_{\D
f}\Rm=\D f*\D\Rm+\D^2f*\Rm$. Alternatively, we have given another derivation of the elliptic equation (\ref{3.ellequ}) in (\ref{2.ellder}). Now, we set $u:=\abs{\Rm}$ and compute
\begin{equation}\label{3.ellipticineq}
\begin{split}
-u\Lap u&=-\tfrac{1}{2}\Lap u^2+\abs{\D u}^2\\
&=-\tfrac{1}{2}\Lap \abs{\Rm}^2+\abs{\D \abs{\Rm}}^2\\
&=-\scal{ \Rm,\Lap\Rm}-\abs{\D\Rm}^2+\abs{\D\abs{\Rm}}^2.
\end{split}
\end{equation}
By equation (\ref{3.ellequ}) and Young's inequality, we can estimate
\begin{equation}
\begin{split}
-\scal{\Rm,\Lap \Rm}&\leq C_3\left(\abs{\Rm}\abs{\D f}\abs{\D \Rm}+\abs{\Rm}^2+\abs{\Rm}^3\right)\\
&\leq\tfrac{1}{10}\abs{\D \Rm}^2+\left(1+\tfrac{10}{4}C_3\abs{\D f}^2\right)C_3\abs{\Rm}^2+C_3\abs{\Rm}^3,
\end{split}
\end{equation}
for some constant $C_3=C_3(n)<\infty$ depending only on the dimension. Finally we use Kato's inequality $\abs{\D\abs{\Rm}}\leq\abs{\D \Rm}$, and the estimate (\ref{2.scalarest}) for $\abs{\D f}$. Putting everything together, we obtain the elliptic inequality
\begin{equation}\label{3.ellin}
-u\Lap u\leq \tfrac{1}{10}\abs{\D u}^2+C_4u^2+C_3u^3
\end{equation}
on $B_r(p)$, where $C_4=C_4(r,n):=(1+\tfrac{5}{8}C_3(r+\sqrt{2n})^2)C_3$.
Given an elliptic inequality like (\ref{3.ellin}) it is well known to PDE-experts that if the $L^{n/2}$-norm of $u$ is sufficiently small on a ball, then one gets $L^\infty$-bounds on a smaller ball, more precisely
\begin{equation}\label{uepsreg}
\norm{u}_{L^{n/2}(B_{\delta}(x))}\leq \eps
\Rightarrow
\sup_{B_{\delta/2}(x)}\abs{u} \leq
\frac{K}{\delta^{2}}\norm{u}_{L^{n/2}(B_{\delta}(x))},
\end{equation}
for some constants $\eps>0$ and $K<\infty$. For convenience of the reader, we sketch the necessary Moser-iteration argument here: To keep the notation reasonably concise let us assume $\delta=1$ and $n=4$, the general case works similarly. Choose a cutoff-function $0\leq\eta\leq 1$ that is $1$ on $B_{3/4}(x)$, has support in $B_{1}(x)$, and satisfies $\abs{\D \eta}\leq 8$. Multiplying (\ref{3.ellin}) by $\eta^2$ and integrating by parts we obtain
\begin{equation}
\tfrac{9}{10}\int_M\eta^2\abs{\D u}^2dV\leq 2\int_M\eta \abs{\D\eta}u\abs{\D u}dV+\int_M \big(C_4 \eta^2 u^2 + C_3\eta^2 u^3\big)dV.
\end{equation}
Dealing with the first term on the right hand side by Young's inequality and absorption, this gives the estimate
\begin{equation}
\tfrac{1}{2}\int_M\eta^2\abs{\D u}^2dV\leq (C_4+160)\int_{B_1} u^2dV + C_3\int_M\eta^2 u^3dV.
\end{equation}
For the last term, using H\"older's inequality, the assumption that the energy on $B_1$ is less than $\eps$, and the Sobolev-inequality, we get
\begin{equation}
\begin{split}
\int_M\eta^2 u^3dV &\leq\left(\int_{B_1} u^2 dV\right)^{\!1/2}\left(\int_M(\eta u)^4 dV\right)^{\!1/2}\\
&\leq \eps C_S^2\int_M\abs{\D(\eta u)}^2 dV\\
&\leq 2\eps C_S^2\int_M\eta^2\abs{\D u}^2 dV +50\eps C_S^2\int_{B_1} u^2 dV,
\end{split}
\end{equation}
where $C_S<\infty$ is the local Sobolev constant on $B_1$. The main idea is that if we choose $\eps$ so small that $2\eps C_S^2C_3\leq\tfrac{1}{4}$ then the $\int \eta^2\abs{\D u}^2$ term can be absorbed, giving
\begin{equation}
\tfrac{1}{4}\int_M \eta^2\abs{\D u}^2dV\leq (C_4+200)\int_{B_1} u^2dV,
\end{equation}
and using the Sobolev inequality we arrive at the $L^4$-estimate
\begin{equation}
\norm{u}_{L^{4}(B_{3/4})}\leq 2C_S\sqrt{C_4+200}\norm{u}_{L^{2}(B_{1})}.
\end{equation}
Now we choose a sequence of radii $r_k=\tfrac{1}{2}+\tfrac{1}{2^k}$ interpolating between $r_1=1$ and $r_{\infty}=\tfrac{1}{2}$. We multiply (\ref{3.ellin}) by $\eta_k^2u^{p_k}$, where $p_k=2^k-2$, and $0\leq\eta_k\leq 1$ is a cutoff function that equals $1$ on $B_{r_{k+1}}$, has support in $B_{r_k}$, and satisfies $\abs{\D \eta_k}\leq 2/(r_k-r_{k+1})$.
Carrying out similar steps as above we obtain the iterative estimates
\begin{equation}
\norm{u}_{L^{2^{k+1}}(B_{r_{k+1}})}\leq C_k\norm{u}_{L^{2^k}(B_{r_k})}.
\end{equation}
The product of the constants $C_k$ converges and sending $k\to\infty$ gives the desired estimate
\begin{equation}
\norm{u}_{L^\infty(B_{1/2})}\leq K\norm{u}_{L^{2}(B_1)}.
\end{equation}
Note that the estimate (\ref{uepsreg}) is of course only useful if we can get uniform constants $\eps>0$ and $K<\infty$ for our sequence of shrinkers. This crucial point is taken care of by Lemma \ref{3.sobconst}, so we indeed get constants $\eps_1(r)=\eps_1(r,n,\ul{\mu})>0$ and $K_0(r)=K_0(r,n,\ul{\mu})<\infty$, such that
\begin{equation}
\norm{\Rm}_{L^{n/2}(B_{\delta}(x))}\leq \eps_1(r)
\Rightarrow
\sup_{B_{\delta/2}(x)}\abs{\Rm} \leq
\frac{K_0(r)}{\delta^{2}}\norm{\Rm}_{L^{n/2}(B_{\delta}(x))},
\end{equation}
for every ball $B_\delta(x)\subset B_r(p), 0<\delta\leq\delta_0(r)$. Once one has $L^\infty$ control, the hard work is done and it is  standard to bootstrap the elliptic equation (\ref{3.ellequ}) to get $C^\infty$ bounds on the ball $B_{\delta/4}(x)$. The only slightly subtle point is that higher derivatives of $f$ appear when differentiating (\ref{3.ellequ}), but one can get rid of them again immediately using the soliton equation (\ref{2.soleqn}). This finishes the proof of the $\eps$-regularity lemma.
\end{proof}

Let us now explain how to finish the proof of Theorem \ref{1.mainthm1}. Let $(M^n_i,g_i,f_i)$ be a sequence of gradient shrinkers satisfying the assumptions of Theorem \ref{1.mainthm1}. By Theorem \ref{2.GHlimit}, we can assume (after passing to a subsequence) that the sequence converges in the pointed Gromov-Hausdorff sense. By passing to another subsequence, we can also assume that the auxiliary constants converge.\\
Let $r<\infty$ large and $0<\delta\leq\delta_1(r,E(r),n,\ul{\mu})$ small enough. The assumption (\ref{1.Rmbound}) gives a uniform bound $E(r)$ for the energy contained in $B_r(p_i)$. So there can be at most $\tfrac{E_1(r)}{\eps_1(r)}$ disjoint $\delta$-balls in $B_r(p_i)$ that contain energy more than $\eps_1(r)$, and away from those bad balls we get $C^\infty$-estimates for the curvatures using the $\eps$-regularity lemma. Recall that we also have volume-noncollapsing by Lemma \ref{2.noncollapslemma}, and that we get $C^\infty_{\textrm{loc}}$ bounds for $f_i$ in regions with bounded curvature, using the elliptic equation
\begin{equation}
\Lap f=\tfrac{n}{2}-R.
\end{equation}
Thus, putting everything together (and playing around with the parameters $r$ and $\delta$ a bit), we can find on any $(M_i,g_i)$ suitable balls $B_\delta(x_i^k(\delta))$, $1\leq k\leq L_i(r)\leq L(r)=L(r,E(r),n,\ul{\mu})$, such that on
\begin{equation}\label{3.defMird}
X_i:={B_{r}(p_i)} \setminus \bigcup_{k=1}^{L_i(r)}B_\delta(x_i^k(\delta)) \subset M_i,
\end{equation}
we have the estimates
\begin{equation}
\begin{split}
\sup_{X_i}\abs{\D^\ell\Rm_{g_i}}&\leq C_\ell(\delta,r,n,\ul{\mu})\\
\sup_{X_i}\abs{\D^\ell f_i}&\leq C_\ell(\delta,r,n,\ul{\mu}).\\
\end{split}
\end{equation}
Together with the volume-noncollapsing, this is exactly what we need to pass to a smooth limit. Thus, sending $r\to\infty$ and $\delta\to 0$ suitably and passing to a diagonal subsequence, we obtain a (possibly incomplete) smooth limit gradient shrinker. Since we already know, that the manifolds $M_i$ converge in the pointed Gromov-Hausdorff sense, this limit can be completed as a metric space by adding locally finitely many points and the convergence is in the sense of Definition \ref{3.defconv}. We have thus proved Theorem \ref{1.mainthm1} up to the statement that the isolated singular points are of orbifold shrinker type.\\

This claimed orbifold structure at the singular points is a local statement, so we can essentially refer to \cite{CS07,Zha06}. Nevertheless, let us sketch the main steps, following Tian \cite[Sec. 3 and 4]{Tia90} closely (see also \cite{And89,BKN89} for similar proofs).\\

Step 1 ($C^0$-multifold): The idea is that blowing up around a singular point will show that the tangent cone is a union of finitely many flat cones over spherical space forms. Improving this a bit, one also gets $C^0$-control over $g$, and thus the structure of a so called $C^0$-multifold (``multi" and not yet ``orbi", since we have to wait until the next step to see that ``a union of finitely many flat cones over spherical space forms" can actually be replaced by ``a single flat cone over a spherical space form"). The precise argument goes as follows: Near an added point $q\in S\subset M_\infty$, we have
\begin{equation}\label{3.c0multi}
\abs{\D^\ell \Rm_{g_\infty}}_{g_\infty}(x)\leq\frac{\eps(\varrho(x))}{\varrho(x)^{2+\ell}},
\end{equation}
where $\varrho(x)=d_\infty(x,q)$. Here and in the following, $\eps(\varrho)$ denotes a quantity that tends to zero for $\varrho\to 0$ and we always assume that $\varrho$ is small enough. With $\eps(\varrho)\to0$ in (\ref{3.c0multi}), together with the Bakry-Emery volume comparison and the non-collapsing, it follows that the tangent cone at $q$ is a finite union of flat cones over spherical space forms $S^{n-1}/\Gamma_\beta$. The tangent cone is unique and by a simple volume argument we get an explicit bound (depending on $r,n,\ul{\mu}$) for the order of the orbifold groups $\Gamma_\beta$ and the number of components $\sharp\{\beta\}$. As in \cite[Lemma 3.6, Eq. (4.1)]{Tia90} there exist a neighborhood $U\subset M_\infty$ of $q$ and for every component $U_\beta$ of $U\setminus \{q\}$ an associated covering $\pi_\beta:B_\varrho^\ast=B_\varrho(0)\setminus\{0\}\to U_\beta$ such that $g^\beta:=\pi_\beta^\ast g_\infty$ can be extended to a $C^0$-metric over the origin with the estimates
\begin{equation}
\begin{split}
&\sup_{B_\varrho^\ast}\abs{g^\beta-g_E}_{g_E}\leq \eps(\varrho),\\
\abs{D^I g^\beta}_{g_E}(x)&\leq\frac{\eps(\varrho(x))}{\varrho(x)^{\abs{I}}},\quad x\in B_\varrho^\ast,\; 1\leq\abs{I}\leq 100,
\end{split}
\end{equation}
where $g_E$ is the Euclidean metric, $D$ the Euclidean derivative and $I$ a multiindex.\\

Step 2 ($C^0$-orbifold): The idea is that if $U\setminus \{q\}$ had two or more components, than all geodesics in an approximating sequence would pass through a very small neck, but this yields a contradiction to the volume comparison theorem. For the precise argument, let $q\in S \subset M_\infty$ be an added point and choose points $x_i\in M_i$ converging to $q$. By the non-collapsing and the Bakry-Emery volume comparison with the bounds for $f_i$ and $\abs{\D f_i}$, there exists a constant $C<\infty$ such that for $\varrho$ small enough any two points in $\partial B_{\varrho}(x_i)$ can be connected by a curve in $B_{\varrho}(x_i)\setminus B_{\varrho/C}(x_i)$ of length less than $C\varrho$. This follows by slightly modifying the proof of \cite[Lemma 1.2]{AndChe91} and \cite[Lemma 1.4]{AbrGro90}. Since the convergence is smooth away from $q$, it follows that $U\setminus \{q\}$ is connected. In particular, the tangent cone at $q$ consists of a single flat cone over a spherical space form $S^{n-1}/\Gamma$.\\

Step 3 ($C^\infty$-orbifold): The final step is to get $C^\infty$ bounds in suitable coordinates. Let $g^1$ be the metric on $B_\varrho^\ast$ from Step 1 and 2. In the case $n=4$, using Uhlenbeck's method for removing finite energy point singularities in the Yang-Mills field \cite{Uhl82}, we get an improved curvature decay
\begin{equation}
\abs{\Rm_{g^1}}_{g^1}(x)\leq\frac{1}{\abs{x}^{\delta}}\label{3.curvdec}
\end{equation}
for $\delta>0$ as small as we want on a small enough punctured ball $B_\varrho^\ast$. The proof goes through almost verbatim as in \cite[Lemma 4.3]{Tia90}. The only difference is that instead of the Yang-Mills equation we use
\begin{equation}
\D_iR_{ijk\ell}=\D_kR_{\ell j}-\D_\ell R_{kj}=-\D_k\D_\ell\D_jf +\D_\ell\D_k\D_jf =
R_{\ell kjp}\D_p f
\end{equation}
and the estimates for $\abs{\D f}$ from Section \ref{SectionGH}. The point is that the Yang-Mills equation $\D_iR_{ijk\ell}=0$ is satisfied up to some lower order term that can be dealt with easily. The case $n\geq 5$ is more elementary, and Sibner's test function \cite{Si85} gives $L^\infty$ bounds for the curvature, in particular (\ref{3.curvdec}) is also satisfied in this case.\\
Due to the improved curvature decay, there exists a diffeomorphism $\psi:B_{\varrho/2}^\ast\to \psi(B^\ast_{\varrho/2})\subset B_\varrho^\ast$ that extends to a homeomorphism over the origin such that $\psi^\ast g^1$ extends to a $C^{1,\alpha}$ metric over the origin (for any $\alpha<1-\delta$). By composing with another diffeomorphism (denoting the composition by $\varphi$), we can assume that the standard coordinates on $B_{\varrho/4}$ are harmonic coordinates for $\varphi^\ast g^1$. Finally, let $\pi:=\pi_1\circ \varphi$, $g:=\pi^\ast g_\infty$ and $f:=\pi^\ast f_\infty$.
Then, for $(g,f)$ we have the elliptic system
\begin{equation}
\begin{split}\label{3.finellsys}
\Lap_g f &=\abs{\D f}_g^2-f+\tfrac{n}{2}-C,\\
\Rc_g &=\tfrac{1}{2}g-\Hess_gf.
\end{split}
\end{equation}
This is indeed elliptic, since $R_{ij}(g)=-\tfrac{1}{2}\sum_k\partial_k\partial_k g_{ij}+Q_{ij}(g,\partial g)$ in harmonic coordinates. It is now standard to bootstrap (\ref{3.finellsys}) starting with the $C^{1,\alpha}$-bound for $g$ and the $C^{0,1}$-bound for $f$ to obtain $C^{\infty}$-bounds for $(g,f)$ and to conclude that $(g,f)$ can be extended to a smooth gradient shrinker over the origin. This finishes the proof of Theorem \ref{1.mainthm1}.\\

\begin{rem}
Every added point is a singular point. Indeed, suppose $\Gamma$ is trivial and $K_i=\abs{\Rm_{g_i}}_{g_i}(x_i)=\max_{B_\varrho(x_i)}\abs{\Rm_{g_i}}_{g_i}\to\infty,\; x_i\to q$ for some subsequence. Then, a subsequence of $(M_i,K_ig_i,x_i)$ converges to a non-flat, Ricci-flat manifold with the same volume ratios as in Euclidean space, a contradiction. 
\end{rem}

\begin{rem}
As discovered by Anderson \cite{And89}, one can use the following two observations to rule out or limit the formation of singularities a priori: For $n$ odd, $\RR P^{n-1}$ is the only nontrivial spherical space form and it does not bound a smooth compact manifold. For $n=4$, every nontrivial orientable Ricci-flat ALE manifold has nonzero second Betti number.
\end{rem}

\section{A local Chern-Gauss-Bonnet argument}\label{SectionGB}

In this section, we prove Theorem \ref{1.mainthm2}. To explain and motivate the Gauss-Bonnet with cutoff argument, we will first prove a weaker version (Proposition \ref{4.propHess}).\\

Recall from Section  \ref{SectionGH} that $f$ grows like one-quarter distance squared, that $R$ and $\abs{\D f}^2$ grow at most quadratically, and that the volume growth is at most Euclidean. These growth estimates will be used frequently in the following.\\

The next lemma, first observed by Munteanu-Sesum \cite{MS09}, will be very
useful in the following. To keep this section self-contained, we give a quick proof.

\begin{lemma}[Weighted $L^2$ estimate for Ricci]\label{4.sesumlemma}
For $\lambda>0$ and $\ul{\mu}>-\infty$ there exist constants
$C(\lambda)=C(\lambda,\ul{\mu},n)<\infty$ such that for every
gradient shrinker $(M^n,g,f)$ with $\mu(g)\geq\ul{\mu}$ and
normalization as before,
\begin{equation}\label{4.sesumtrick}
\int_M \abs{\Rc}^2 e^{-\lambda f}dV \leq C(\lambda) < \infty.
\end{equation}
\end{lemma}

\begin{proof}
Take a cutoff function $\eta$ as in Section \ref{SectionGH} and set
$\eta_r(x)=\eta(d(x,p)/r)$. Note that $\divop(e^{-f}\Rc)=0$ by
(\ref{2.bianchi}). Using this, the soliton equation, a partial
integration and the inequality $ab\leq a^2/4+b^2$, we compute
\begin{align*}
\int_M \eta_r^2\abs{\Rc}^2 e^{-\lambda f}dV &= \int_M \eta_r^2
\scal{\tfrac{1}{2}g-\D^2 f,\Rc} e^{-\lambda f}dV\\
&=\int_M \big(\tfrac{1}{2}\eta_r^2R +(1-\lambda)\eta_r^2\Rc(\D f,\D f)
+2\eta_r \Rc(\D\eta_r,\D f)\big)e^{-\lambda f}dV\\
&\leq \tfrac{1}{2}\int_M \eta_r^2 \abs{\Rc}^2 e^{-\lambda f}dV
+ \int_M \eta_r^2 \big(\tfrac{1}{2}R+(1-\lambda)^2\abs{\D
f}^4\big)e^{-\lambda f}dV\\
&\quad + 4 \int_M\abs{\D\eta_r}^2\abs{\D f}^2e^{-\lambda f}dV.
\end{align*}
The first term can be absorbed. The second term is uniformly bounded
and the last term converges to zero as $r\to \infty$ by the growth
estimates from Section \ref{SectionGH}.
\end{proof}

As a consequence of Lemma \ref{4.sesumlemma}, we can replace the
Riemann energy bound in Theorem \ref{1.mainthm1} by a Weyl energy
bound in dimension four.

\begin{cor}[Weyl implies Riemann energy condition]
Every sequence of $4$-dimensional gradient shrinkers $(M_i,g_i,f_i)$
(with normalization and basepoint as usual) with entropy bounded
below, $\mu(g_i)\geq \ul\mu$, and a local Weyl energy bound
\begin{equation}
\int_{B_r(p_i)}\abs{W_{g_i}}_{g_i}^2dV_{g_i} \leq C(r) < \infty,
\quad \forall i,r
\end{equation}
satisfies the energy condition (\ref{1.Rmbound}).
\end{cor}

\begin{rem}
As a consistency check, note that in dimension $n=3$, $\Rm$ is
determined by $\Rc$ and thus only a lower bound for the entropy is
needed and the limit is smooth. Of course, the existence of a smooth
limit also follows from Theorem \ref{2.smoothlimit} and the fact
that $\Rm\geq 0$ on gradient shrinkers for $n=3$. All this is not
surprising, since the only $3$-dimensional gradient shrinkers are
the Gaussian soliton, the cylinder, the sphere and quotients thereof
\cite{Cao09}.\\
\end{rem}

In the following, the goal is to get local energy bounds from
4d-Gauss-Bonnet with boundary. For a $4$-manifold $N$ with boundary
$\partial N$, the Chern-Gauss-Bonnet formula says (see e.g.
\cite{Gil84})
\begin{equation}\label{4.GaussBonnet}
\begin{split}
32\pi^2\chi(N)&=\int_{N}\big(\abs{\Rm}^2-4\abs{\Rc}^2+R^2\big)dV\\
&\quad +16\int_{\partial N}k_1k_2k_3\,dA +8\int_{\partial
N}\big(k_1K_{23}+k_2K_{13}+k_{3}K_{12}\big)dA,
\end{split}
\end{equation}
where the $k_i=\ff (e_i,e_i)$ are the principal curvatures of
$\partial N$ (here $e_1,e_2,e_3$ is an orthonormal basis of
$T\partial N$ diagonalizing the second fundamental form) and the
$K_{ij}=\Rm(e_i,e_j,e_i,e_j)$ are sectional curvatures of $N$.\\

In a first step, we prove Theorem \ref{1.mainthm2} under an extra
assumption which ensures in particular that the cubic boundary term
has the good sign.

\begin{prop}[Convexity implies Riemann energy condition]\label{4.propHess}
Every sequence of $4$-dimensional gradient shrinkers $(M_i,g_i,f_i)$
(with normalization and basepoint as usual) with entropy bounded
below, $\mu(g_i)\geq \ul\mu$, Euler characteristic bounded above,
$\chi(M_i)\leq\ol\chi$, and convex potential at large distances,
\begin{equation}\label{4.convexpot}
\Hess_{g_i}f_i(x)\geq 0\qquad \text{if } d(x,p_i)\geq r_0,
\end{equation}
satisfies the energy condition (\ref{1.Rmbound}).
\end{prop}

\begin{proof}
Let us introduce some notation first. We suppress the index $i$ and
write $F(x)=e^{-f(x)}$ and define the level and superlevel sets
\begin{equation}
\Sigma_u=\{x\in M \mid F(x)=u\}, \quad M_u=\{x\in M \mid F(x) \geq
u\}.
\end{equation}
Note that $M_0=M$ and $M_{u_2}\subset M_{u_1}$ if $u_2\geq
u_1$.\\

By the traced soliton equation (\ref{2.tracedsoleqn}) and assumption
(\ref{4.convexpot}), we have $R\leq\tfrac{n}{2}$ at large distances.
Using this, the auxiliary equation (\ref{2.auxiliaryeqn}), Lemma
\ref{2.growthlemma}, and the bounds $\ul{\mu}\leq -C_1(g)\leq
\ol{\mu}$, we see that $f$ does not have critical points at large
distances. In fact, there is a constant $u_0=u_0(r_0,\ul\mu)>0$ such
that $\abs{\D f}\geq 1$ and $\D^2f\geq 0$ if $F(x)\leq u_0$.
Moreover, for $0<u\leq u_0$ the $\Sigma_u$ are smooth compact
hypersurfaces, they are all diffeomorphic and we have $\partial
M_u=\Sigma_u$ and
$\chi(M_u)=\chi(M)$.\\

Define a cutoff function $\vartheta(x):=\min\{u_0,F(x)\}$, then
\begin{equation}\label{4.fubinitrick}
\int_M\abs{\Rm}^2 \vartheta\, dV=\int_M\abs{\Rm}^2\int_0^{u_0}
1_{\{u\leq F\}}\, du\, dV=\int_0^{u_0} \int_{M_u}\abs{\Rm}^2 dV du.
\end{equation}
Now, we can apply (\ref{4.GaussBonnet}) for $N=M_u$. Note that
$\chi(M_u)\leq \ol\chi$, and that the scalar curvature term and the
cubic boundary term are nonnegative. Indeed,
\begin{equation}
\ff=-\D^2_{\perp} F/\abs{\D F}=\tfrac{1}{\abs{\D f}}\left(\D^2 f-\D
f \otimes \D f\right)_{\perp}=\tfrac{1}{\abs{\D f}}\D^2_{\perp}f\geq
0,
\end{equation}
where $\perp$ denotes the restriction of the Hessian to $T\Sigma_u$.
Thus, we obtain
\begin{equation}
\begin{split}
\int_{M_u}\abs{\Rm}^2dV &\leq 32\pi^2\ol\chi
+4\int_{M_u}\abs{\Rc}^2dV\\
&\quad -8\int_{\Sigma_u}(k_1K_{23}+k_2K_{13}+k_{3}K_{12})dA
\end{split}
\end{equation}
and undoing (\ref{4.fubinitrick}), using $\vartheta\leq e^{-f}$,
$\abs{k_i}\leq \abs{\ff}$ and $\abs{K_{ij}}\leq\abs{\Rm}$, this
implies
\begin{align}
\int_M\abs{\Rm}^2\vartheta\, dV\leq 32\pi^2\ol\chi
u_0+4\int_M\abs{\Rc}^2e^{-f}dV
+24\int_0^{u_0}\int_{\Sigma_u}\abs{\ff}\abs{\Rm}dAdu.
\end{align}
The Ricci term can be estimated as in (\ref{4.sesumtrick}). For the
last term we use the coarea formula (observe the cancelation):
\begin{equation}
\begin{split}
\int_0^{u_0}\int_{\Sigma_u}\abs{\ff}\abs{\Rm}dA du
&\leq\int_{M\setminus M_{u_0}}\frac{\abs{\D^2f}}{\abs{\D f}}
\abs{\Rm}\abs{\D f}\vartheta\, dV\\
&\leq\tfrac{1}{48}\int_M\abs{\Rm}^2 \vartheta\, dV+12\int_M
\abs{\D^2 f}^2e^{-f}dV.
\end{split}
\end{equation}
The first term can be absorbed, the second one can be dealt with as
in~(\ref{4.sesumtrick}),
\begin{equation}
\begin{split}
\int_M\abs{\D^2 f}^2e^{-f}dV &=\int_M\scal{\D^2
f,\tfrac{1}{2}g-\Rc}e^{-f}dV\\
&=\tfrac{1}{2}\int_M\Lap f e^{-f}dV+\int_M\langle\D f,\divop(e^{-f}\Rc)\rangle dV\\
&=\tfrac{1}{2}
\int_M\left(\tfrac{n}{2}-R\right) e^{-f}dV,
\end{split}
\end{equation}
where we used the traced soliton equation and $\divop(e^{-f}\Rc)=0$ in the last step.
Putting everything together, we obtain a uniform bound for $\int_M
\abs{\Rm}^2 \vartheta\, dV$, and (\ref{1.Rmbound}) follows.
\end{proof}

Let us now replace the (unnatural) assumption (\ref{4.convexpot}) by
the weaker assumption (\ref{1.noncritf}). Let
$u_0=u_0(r_0,\ul\mu)>0$ such that $\abs{\D f}\geq c$ if $F(x)\leq
u_0$ and $\vartheta(x):=\min\{u_0,F(x)\}$ a cutoff function as
before. The proof is essentially identical, except that in addition
we have to estimate (the negative part of) the cubic boundary term
in the Gauss-Bonnet formula. By the coarea formula
\begin{equation}\label{4.cubicest}
\Abs{\int_0^{u_0}\int_{\Sigma_u}\det\ff dAdu}\leq \int_{M\setminus
M_{u_0}}\frac{\abs{\D^2_\perp f}^3}{\abs{\D f}^2}e^{-f}dV\leq\frac{1}{c^2}
\int_M\abs{\Rc-\tfrac{1}{2}g}^3e^{-f}dV.
\end{equation}
Note that the only difficult term is $\int_M\abs{\Rc}^3e^{-f}dV$,
since all other terms can be uniformly bounded using Lemma
\ref{4.sesumlemma}. Fortunately, we can bound this weighted
$L^3$-norm of Ricci by uniformly controlled terms and a weighted
$L^2$ Riemann term that can be absorbed in the Gauss-Bonnet
argument. Exploiting the algebraic structure of the equations for
gradient shrinkers and the full strength of Lemma
\ref{4.sesumlemma}, we obtain the following key estimate.

\begin{lemma}[Weighted $L^3$ estimate for Ricci]\label{4.RcRmlemma}
For $\eps>0$ and $\ul{\mu}>-\infty$ there exist constants
$C(\eps)=C(\eps,\ul{\mu},n)<\infty$ such that for every gradient
shrinker $(M^n,g,f)$ with our usual normalization and
$\mu(g)\geq\ul{\mu}$ we have the estimate
\begin{equation}\label{4.RcRmestimate}
\int_M \abs{\Rc}^3e^{-f}dV \leq \eps \int_M \abs{Rm}^2e^{-f}dV
+C(\eps).
\end{equation}
\end{lemma}

\begin{proof}
Analogous to (\ref{2.bianchi}), we have
\begin{equation}\label{4.bianchi}
\D_kR_{ij}-\D_iR_{kj}=-\D_k\D_i\D_jf +\D_i\D_k\D_jf =
R_{ikj\ell}\D_\ell f
\end{equation}
and as a direct consequence $\divop(e^{-f}\Rm)=0$. Moreover, analogous to (\ref{2.ellipteqR}), the shrinker
version of the evolution equation for the Ricci tensor is
\begin{equation}\label{4.ellipticRc}
R_{ij}+\scal{\D f,\D R_{ij}} = \Lap R_{ij} + 2R_{ikj\ell}R_{k\ell}.
\end{equation}
Now, for a cutoff function $\eta_r$ as in the proof of Lemma
\ref{4.sesumlemma}, we compute
\begin{align*}
\int_M \eta_r\abs{\Rc}^3e^{-f}dV &=\int_M \eta_r
\abs{\Rc}\scal{\tfrac{1}{2}g-\D^2f,\Rc}e^{-f}dV\\
&=\int_M \big(\tfrac{1}{2}\eta_r \abs{\Rc}R+
\abs{\Rc}\Rc(\D f,\D\eta_r) +\eta_r\Rc(\D f,\D\abs{\Rc})\big)e^{-f}dV\\
&\leq \int_M \big(\tfrac{1}{2}\eta_r\abs{\Rc}R+
\abs{\D\eta_r}\abs{\Rc}^2\abs{\D f}\big)e^{-f}dV\\
&\quad + \delta \int_M \eta_r\abs{\D\Rc}^2e^{-\frac{3}{2}f}dV
+\tfrac{1}{4\delta}\int_M \eta_r \abs{\Rc}^2\abs{\D
f}^2e^{-\frac{1}{2}f}dV\\
&\leq \delta \int_M \eta_r\abs{\D\Rc}^2e^{-\frac{3}{2}f}dV
+C(\delta),
\end{align*}
for $\delta>0$ to be chosen later. Here, we used Young's
inequality, Kato's inequality, the growth estimates from Section \ref{SectionGH}
and Lemma \ref{4.sesumlemma} (note that $\abs{\D f}^2e^{-f/2}\leq
Ce^{-f/4}$ etc.). Note that the constant $C(\delta)$ does not depend on the
scaling factor $r$ of the cutoff function $\eta_r$, so by sending
$r\to\infty$, we obtain
\begin{equation}\label{4.Rc3DRc}
\int_M \abs{\Rc}^3e^{-f}dV \leq \delta \int_M
\abs{\D\Rc}^2e^{-\frac{3}{2}f}dV +C(\delta).
\end{equation}
Next, we estimate the weighted $L^2$-norm of $\D\Rc$ with a partial
integration, equation (\ref{4.ellipticRc}), and Young's inequality,
\begin{align*}
\int_M \eta_r^2\abs{\D\Rc}^2e^{-\frac{3}{2}f}dV &= - \int_M
\eta_r^2\big(\Lap R_{ij}-\tfrac{3}{2}\scal{\D f,\D
R_{ij}}\big)R_{ij}e^{-\frac{3}{2}f}dV\\
&\quad - \int_M 2\eta_r\scal{\D\eta_r,\D
R_{ij}}R_{ij}e^{-\frac{3}{2}f}dV\\
&\leq -\int_M \eta_r^2\big(R_{ij}-2R_{ikj\ell}R_{k\ell})R_{ij}
e^{-\frac{3}{2}f}dV\\
&\quad + \int_M \big(\tfrac{1}{2}\eta_r^2\abs{\D\Rc}^2+\tfrac{1}{4}\eta_r^2 \abs{\Rc}^2\abs{\D f}^2 +
4\abs{\D\eta_r}^2\abs{\Rc}^2
\big)e^{-\frac{3}{2}f}dV.
\end{align*}
By absorption, the growth estimates from Section \ref{SectionGH},
Lemma \ref{4.sesumlemma} and the soliton equation, we obtain
\begin{align*}
\int_M \eta_r^2\abs{\D\Rc}^2e^{-\frac{3}{2}f}dV &\leq C-2\int_M
\eta_r^2\big(R_{ij}-2R_{ikj\ell}R_{k\ell})R_{ij} e^{-\frac{3}{2}f}dV\\
& = C-4\int_M \eta_r^2 R_{ikj\ell}R_{ij}\D_k\D_\ell f
e^{-\frac{3}{2}f}dV.
\end{align*}
Finally, with another partial integration, $\divop(e^{-f}\Rm)=0$, and
with
\begin{equation}
2R_{ikj\ell}\D_kR_{ij}\D_\ell f = R_{ikj\ell}\D_\ell
f(\D_kR_{ij}-\D_iR_{kj})= \abs{R_{ikj\ell}\D_\ell f}^2,
\end{equation}
which follows from (\ref{4.bianchi}) and which is the identity that
makes the proof work, we get
\begin{align}\label{4.DRcRm}
\int_M \eta_r^2\abs{\D\Rc}^2e^{-\frac{3}{2}f}dV &\leq C+2\int_M
\eta_r^2 \abs{R_{ikj\ell}\D_\ell f}^2 e^{-\frac{3}{2}f}dV\nonumber\\
&\quad + \int_M \big(8\eta_r\D_k\eta_r R_{ikj\ell}R_{ij}\D_\ell f
-2\eta_r^2 R_{ikj\ell}R_{ij}\D_k f\D_\ell f\big)e^{-\frac{3}{2}f}dV\nonumber\\
&\leq C + C\int_M \eta_r^2 \abs{\Rm}^2e^{-f}dV.
\end{align}
In the last step, we used again Young's inequality, the growth
estimates from Section \ref{SectionGH}, Lemma \ref{4.sesumlemma} and
$\abs{\D f}^2e^{-\frac{3}{2}f}\leq Ce^{-f}$. The claim now follows
by sending $r\to\infty$, plugging into (\ref{4.Rc3DRc}), and
choosing $\delta>0$ such that $C\delta \leq \eps$ for the constant
$C$ in (\ref{4.DRcRm}).
\end{proof}

Now Theorem \ref{1.mainthm2} is an immediate consequence.

\begin{proof}[Proof of Theorem \ref{1.mainthm2}]
Picking $\eps=\eps(c,r_0,\ul{\mu})>0$ so small that $\eps e^{-f}\leq\vartheta c^2/100$ and applying Lemma \ref{4.RcRmlemma}, the theorem
follows as explained in the discussion after Proposition
\ref{4.propHess}.
\end{proof}

\appendix

\section{Proofs of the lemmas from Section 2}\label{applemmas}

\begin{proof}[Proof of Lemma \ref{2.growthlemma}]
From (\ref{2.auxiliaryeqn}) and (\ref{2.scalarpos}), we obtain
\begin{equation}
0\leq \abs{\D f}^2\leq f+C_1,\label{2.dfbound}
\end{equation}
i.e.\ $\abs{\D\sqrt{f+C_1}}\leq \frac{1}{2}$ whenever $f+C_1>0$.
Hence $\sqrt{f+C_1}$ is $\frac{1}{2}$-Lipschitz and thus
\begin{equation}
\sqrt{f(x)+C_1}\leq \tfrac{1}{2}\big(d(x,y)
+2\sqrt{f(y)+C_1}\big),\label{2.lipschitz}
\end{equation}
for all $x,y\in M$, which will give the upper bound in
(\ref{2.quadgrowth}). The idea to prove the lower bound is the same as in the theorem of Myers (which would give a diameter bound if the shrinker potential was constant). Consider a minimizing geodesic
$\gamma(s)$, $0\leq s\leq s_0:=d(x,y)$, joining $x=\gamma(0)$ with
$y=\gamma(s_0)$. Assume $s_0>2$ and let
\begin{equation*}
\phi(s)=\begin{cases} s,&s\in[0,1]\\
1,&s\in[1,s_0-1]\\
s_0-s,&s\in[s_0-1,s_0].\end{cases}
\end{equation*}
By the second variation formula for the energy of $\gamma$,
\begin{equation*}
\int_0^{s_0}\phi^2\Rc(\gamma',\gamma')ds \leq
(n-1)\int_0^{s_0}\phi'^2 ds=2n-2,
\end{equation*}
where $\gamma'(s)=\ds \gamma(s)$. Note that by the soliton equation (\ref{2.soleqn})
\begin{equation*}
\Rc(\gamma',\gamma')=\tfrac{1}{2}-\D_{\gamma'}\D_{\gamma'}f,
\end{equation*}
which implies
\begin{equation}\label{2.distancebound}
\begin{split}
\tfrac{d(x,y)}{2}+\tfrac{4}{3}-2n &\leq
\int_0^{s_0}\phi^2\D_{\gamma'}\D_{\gamma'}f ds\\
&=-2\int_{0}^{1}\phi \D_{\gamma'}fds+2\int_{s_0-1}^{s_0}\phi
\D_{\gamma'}f ds\\
&\leq\sup_{s\in[0,1]}\abs{\D_{\gamma'}f}
+\sup_{s\in[s_0-1,s_0]}\abs{\D_{\gamma'}f}\\
&\leq\sqrt{f(x)+C_1}+\tfrac{1}{2}+\sqrt{f(y)+C_1}+\tfrac{1}{2},
\end{split}
\end{equation}
where we used (\ref{2.dfbound}) and the fact that $\sqrt{f+C_1}$ is
$\frac{1}{2}$-Lipschitz in the last step. By
(\ref{2.distancebound}), every minimizing sequence is bounded and
$f$ attains its infimum at a point $p$. Since $\Lap f(p)\geq 0$,
(\ref{2.tracedsoleqn}) and (\ref{2.scalarpos}) imply
\begin{equation}
0\leq R(p)\leq\tfrac{n}{2}.
\end{equation}
Using this and $\D f(p)=0$, equation (\ref{2.auxiliaryeqn}) implies
\begin{equation}
0 \leq f(p)+C_1\leq \tfrac{n}{2}.\label{2.fcest}
\end{equation}
Now the quadratic growth estimate (\ref{2.quadgrowth}) follows from
(\ref{2.lipschitz}), (\ref{2.distancebound}) and (\ref{2.fcest}) by
setting $y=p$. Finally, if $d(x,p)>5n+\sqrt{2n}$, then
\begin{equation*}
f(x)+C_1\geq \tfrac{1}{4}(d(x,p)-5n)^2>\tfrac{n}{2}\geq f(p)+C_1,
\end{equation*}
which implies the last statement of the lemma.
\end{proof}

\begin{proof}[Proof of Lemma \ref{2.ballslemma}]
Let $\varrho(x)=2\sqrt{f(x)+C_1}$. This grows linearly, since
(\ref{2.quadgrowth}) implies
\begin{equation}\label{2.growthrho}
d(x,p)-5n \leq \varrho(x) \leq d(x,p)+5n.
\end{equation}
Define $\varrho$-discs by $D(r):=\{x\in M\mid\varrho(x)<r\}$, let
$V(r)$ be their volume and
\begin{equation}
S(r):=\int_{D(r)}R\,dV
\end{equation}
their total scalar curvature. Since $\int_{D(r)}\Lap f\,
dV=\int_{\partial D(r)}\abs{\D f}\,dA\geq 0$, integrating
(\ref{2.tracedsoleqn}) gives
\begin{equation}
S(r)\leq \frac{n}{2}V(r),
\end{equation}
i.e.\ the average scalar curvature is bounded by $\tfrac{n}{2}$.
Moreover, (\ref{2.tracedsoleqn}) and (\ref{2.auxiliaryeqn}) imply
\begin{equation*}
(r^{-n}V(r))'=4r^{-(n+2)}S'(r)-2r^{-(n+1)}S(r),
\end{equation*}
which yields the following estimate by integration
\begin{equation}
V(r)\leq\frac{V(r_0)}{r_0^n}r^n+\frac{4}{r^2}S(r)
\end{equation}
for $r\geq r_0:=\sqrt{2n+4}$, see \cite[Eq.~(3)]{Mun09} or
\cite[Eq.~(3.6)]{CZ09} for details. Hence, if $r\geq \sqrt{4n}$, we
get by absorption
\begin{equation*}
V(r)\leq \frac{2V(r_0)}{r_0^n}r^n.
\end{equation*}
Thus, for every $r\geq 5n$ we obtain
\begin{equation*}
\Vol B_r(p)\leq V(r+5n)\leq V(2r) \leq
\tfrac{2^{n+1}}{r_0^n}V(r_0)r^n \leq \tfrac{2^{n+1}}{r_0^n}\Vol
B_{r_0+5n}(p)r^n.
\end{equation*}
This proves the lemma up to the statement that $C_2$ depends only on
the dimension and that (\ref{2.ballseqn}) also holds for balls with $r<5n$. To get this, note that $\abs{\D
f(x)}\leq\tfrac{1}{2}r_0+5n=:a$ for $d(x,p)\leq r_0+5n=:R_0$.
Now, using the fact that the Bakry-Emery Ricci tensor
$\Rc_f=\Rc+\Hess f$ of the manifold with density
$(M,g,e^{-(f+C_1)}dV)$ is nonnegative by the soliton equation, we
obtain, see \cite[Thm.~1.2a]{WW09},
\begin{equation}
\frac{\int_{B_R(p)}e^{-(f+C_1)}dV}{\int_{B_\eps(p)}e^{-(f+C_1)}dV}\leq
e^{aR}\frac{R^n}{\eps^n}.
\end{equation}
for $0<\eps< R\leq R_0$. Since $\abs{f+C_1}\leq a^2$ on $B_{R_0}(p)$, this
implies
\begin{equation}\label{2.BEvolumecomp}
\Vol B_R(p)\leq e^{2a^2+aR}\frac{R^n}{\eps^n}\Vol B_\eps(p)
\end{equation}
and by sending $\eps$ to zero the claim follows.
\end{proof}

\begin{proof}[Proof of Lemma \ref{2.noncollapslemma}]
Suppose towards a contradiction that
there exist a sequence of gradient shrinkers $(M_i,g_i,f_i)$ with
$\mu(g_i)\geq\ul{\mu}$ and balls $B_{\delta_i}(x_i)\subset B_r(p_i)$
with $\delta_i^{-n}\Vol B_{\delta_i}(x_i)\to 0$. We will not
directly use $B_{\delta_i}(x_i)$ but consider the sequence of unit
balls $B_1(x_i)\subset B_{r+1}(p_i)$ instead, which allows to work
with the shrinker entropy as defined above rather than with a
version that explicitly involves a scaling or time parameter $\tau$
as it is necessary for the argument in \cite{KL08}. Set
$a:=\frac{1}{2}(r+1+5n)$, then $\abs{\D f_i}\leq a$,
$\abs{f_i+C_1(g_i)}\leq a^2$ and $R_{g_i}\leq a^2$ on
$B_{r+1}(p_i)$. The volume comparison theorem for the Bakry-Emery
Ricci tensor implies, as in (\ref{2.BEvolumecomp}),
\begin{equation}\label{2.smallvolume}
\Vol B_1(x_i) \leq e^{2a^2+a}\delta_i^{-n}\Vol B_{\delta_i}(x_i)\to
0
\end{equation}
for $i\to\infty$, as well as
\begin{equation}\label{2.smallratio}
\Vol B_1(x_i) \leq 2^n e^{2a^2+a} \Vol B_{1/2}(x_i), \quad
\forall i \in \NN.
\end{equation}
Define the test functions $\tilde{u}_i = c_i^{-1/2}\eta_i$ with
$\eta_i(x)=\eta(d(x,x_i))$ for a cutoff function $\eta$ as in Section \ref{SectionGH} and
with $\int_{M_i} \tilde{u}_i^2 dV = (4\pi)^{n/2}$, i.e.
\begin{equation*}
c_i = (4\pi)^{-n/2}\int_{M_i} \eta_i^2 dV \leq (4\pi)^{-n/2}\Vol
B_1(x_i)\to 0.
\end{equation*}
Let $C$ be an upper bound for $4\eta'^2-\eta^2\log \eta^2$. Using
(\ref{2.smallratio}) and $c_i^{-1}\Vol B_{1/2}(x_i)\leq c_i^{-1} 
\int_{M_i} \eta_i^2 = (4\pi)^{n/2}$, we obtain
\begin{align*}
c_i^{-1}\int_{M_i} \Big(4\abs{\D\eta_i}^2-\eta_i^2\log \eta_i^2\Big)dV
\leq c_i^{-1}\Vol B_1(x_i)C
\leq (4\pi)^{n/2}2^n e^{2a^2+a}C.
\end{align*}
Hence
\begin{align*}
\sW(g_i,\tilde{u}_i) &= (4\pi)^{-n/2}c_i^{-1}\int_{M_i}
\Big(4\abs{\D\eta_i}^2-\eta_i^2\log \eta_i^2\Big)dV\\
&\quad + (4\pi)^{-n/2}\int_{M_i} (R-n +\log c_i)\tilde{u}_i^2 dV\\
&\leq 2^n e^{2a^2+a}C + a^2-n +\log c_i,
\end{align*}
which tends to $-\infty$ as $c_i$ tends to zero, contradicting the
lower entropy bound $\sW(g_i,\tilde{u}_i)\geq \mu(g_i)\geq
\ul{\mu}>-\infty$.
\end{proof}

\makeatletter
\def\@listi{%
  \itemsep=0pt
  \parsep=1pt
  \topsep=1pt}
\makeatother
{\fontsize{10}{11}\selectfont
}
\vspace{10mm}

Robert Haslhofer\\
{\sc ETH Z\"urich, 8092 Z\"urich, Switzerland}\\

Reto M\"uller\\
{\sc Scuola Normale Superiore di Pisa, 56126 Pisa, Italy}
\end{document}